\documentclass[12pt]{amsart}
\usepackage{amsmath}
\usepackage{amssymb}
\usepackage[all,cmtip]{xy}
\usepackage{amsthm}

\DeclareMathAlphabet{\mathpzc}{OT1}{pzc}{m}{it}
\newcommand{\ncom}{\newcommand}

\ncom{\rar}{\rightarrow}
\ncom{\imply}{\Rightarrow}
\ncom{\lrar}{\longrightarrow}
\ncom{\into}{\hookrightarrow}
\ncom{\onto}{\twoheadrightarrow}
\ncom{\ov}{\overline}
\ncom{\m}{\mbox}
\ncom{\sta}{\stackrel}
\ncom{\invlim}{\varprojlim}
\ncom{\xhat}{\widehat}

\ncom{\vspc}{\vspace{3mm}}
\ncom{\End}{{\cE}nd}
\ncom{\tensor}{\otimes}

\ncom{\al}{\alpha}
\ncom{\cHom}{{\mathcal Hom}}

\ncom{\A}{{\mathbb A}}
\ncom{\comx}{{\mathbb C}}
\ncom{\E}{{\mathbb E}}
\ncom{\F}{{\mathbb F}}
\ncom{\G}{{\mathbb G}}
\ncom{\K}{{\mathbb K}}
\ncom{\Le}{{\mathbb L}}
\ncom{\N}{{\mathbb N}}

\ncom{\p}{{\mathbb P}}
\ncom{\Q}{{\mathbb Q}}
\ncom{\R}{{\mathbb R}}
\ncom{\Z}{{\mathbb Z}}

\ncom{\f}{\dfrac}

\ncom{\wtil}{\widetilde}

\ncom{\ci}{{\mathpzc i}}

\ncom{\cA}{{\mathcal A}}
\ncom{\cC}{{\mathcal C}}
\ncom{\cE}{{\mathcal E}}
\ncom{\cF}{{\mathcal F}}
\ncom{\cG}{{\mathcal G}}
\ncom{\cH}{{\mathcal H}}
\ncom{\cI}{{\mathcal I}}
\ncom{\cJ}{{\mathcal J}}
\ncom{\cL}{{\mathcal L}}
\ncom{\cM}{{\mathcal M}}
\ncom{\cN}{{\mathcal N}}
\ncom{\cO}{{\mathcal O}}
\ncom{\cP}{{\mathcal P}}
\ncom{\cQ}{{\mathcal Q}}
\ncom{\cR}{{\mathcal R}}
\ncom{\cS}{{\mathcal S}}
\ncom{\cT}{{\mathcal T}}
\ncom{\cU}{{\mathcal U}}
\ncom{\cV}{{\mathcal V}}
\ncom{\cW}{{\mathcal W}}
\ncom{\cX}{{\mathcal X}}
\ncom{\cY}{{\mathcal Y}}
\ncom{\cZ}{{\mathcal Z}}

\ncom{\cSU}{{\mathcal S \mathcal U}}
\ncom{\eop}{{\hfill $\Box$}}
\ncom{\isom}{\cong}

\ncom{\todo}{{\textbf{TODO}}}
\ncom{\fka}{{\mathfrak{a}}}

\newtheorem{theorem}{Theorem}[section]
\newtheorem{lemma}[theorem]{Lemma}
\newtheorem{corollary}[theorem]{Corollary}

\newtheorem*{theorem*}{Theorem}

\newtheorem{question}[theorem]{Question}

\newtheorem{answer}[theorem]{Answer}

\begin{document}
\baselineskip=16pt

\title[]{A result on intersecting families with maximum transversal size}
\author{Amit Tripathi}
\address{Statistics and Mathematics unit, 
Indian Statistical Institute, Bangalore - 560 059, India}
\email{amittr@gmail.com}
\thanks{This work was supported by a postdoctoral fellowship from NBHM, Department of Atomic Energy.}
\begin{abstract}
We construct an intersecting $k$-family of transversal size  $\lceil \frac{k+1}{2} \rceil$ and length $k+1$ and study some of its properties. We then use this family to prove that $q(4) = 9$. We also construct a $k$-family for $k = 2^m - 1$ of length $2k+1$ and transversal size at least $(2k+1)/3$. 
\end{abstract}
\maketitle

\section*{}

Let $k \in \Z^+$. A $k$-\textit{set} is a set with $k$ elements. A $k$-\textit{family} is a collection of $k$-sets. A $k$-family $\cF$ is \textit{intersecting} if $F \cap G \neq \emptyset$ for all $F, G \in \cF$. A set $C$ is called a \textit{covering set} of $\cF$ if $C \cap F \neq \emptyset$ for all $F \in \cF$. The transversal size $\tau$ of a $k$-family $\cF$ is the smallest possible integer $t$ such that there exists a $t$-set covering $\cF$. A \textit{transversal} is a covering set of transversal size.

Erd\"{o}s and Lov\'{a}sz  \cite{Er-Lov} defined $q(k)$ to be the smallest integer such that there exists a $k$-intersecting family of transversal size $k$ and size $q(k)$. They asked for an estimate on the size of $q(k)$ and in particular to "prove or disprove if $q(k) \sim o(k)$". For a long time this remained open and was one of the favorite problems of Erd\"{o}s, see \cite{Er}. It was settled in affirmative by Jeff Kahn in \cite{Jeff}. Unfortunately, the proof didn't give any estimate for the proportionality constant or an estimate on $k$ from where $q(k)$ becomes a linear function. 

It is easy to see that $q(2) = 3$. In \cite{Frankl}, it was proved (among other things) that $q(3) = 6$. In this note, we construct an intersecting $k$-family $\cM_k$ for all $k$ and study some of its properties. We then use it to prove: 
\begin{theorem*}[See theorem \ref{main_result} and the example given at the end]
q(4) = 9
\end{theorem*}

We further use $\cM_k$ to construct a uniform intersecting $k$-family of transversal size at least $(2k+1)/3$ and length $2k+1$ when $k = 2^m - 1, \,\, \forall\,m \in \N$.

\section{An intersecting family}

The set $\cup_{F \in \cF} F$ is called the \textit{set of vertices} of $\cF$. \textit{Degree of a vertex} is the number of $k$-sets of $\cF$ containing that vertex. \textit{Length or size} of a family $\cF$ is the number of blocks in it. Any $k$-set in the family is called a \textit{block} of the family.

\newpage 
Following lemma constructs the uniform intersecting family $\cM_k$ for every $k$:

\begin{lemma} \label{lemma_minimal} For any $k \in \Z^+$, there exists an intersecting family $\cM_k$ of transversal size $\lceil \frac{k+1}{2} \rceil$ and length $k+1$. Furthermore, the number of vertices in this family is equal to  $k(k+1)/2$.
\end{lemma}
\begin{proof} Let the first $k-$set be $F_1 = \{1,2,\ldots k\}$. Assume we have chosen $F_m$. To construct $F_{m+1}$, we pick exactly one vertex from every $F_i, \,\,$ $i \leq m$ which has appeared only once so far in the family. This gives us $m$ vertices for $F_{m+1}$. We complete this set by choosing any $k-m$ formal symbols which have not appeared so far in the vertices. It is easy to see that this algorithm ends at $k+1$'th step, which is the length of the family.

By construction this is an intersecting $k$-family. The claim on number of vertices follows easily as in $m$'th step we are picking $k-m$ new symbols. 

Finally, all the vertices have degree 2. Suppose a $t$-set $\{x_1, \ldots x_t\}$ covers this family. Then we must have $2t \geq k+1$. Since $t$ must be an integer, we get the claim on transversal size.
\end{proof}

We will need the following uniqueness result.
\begin{corollary} \label{cor_uniqueness} If in the statement of lemma \ref{lemma_minimal}, in addition to the assumptions on length and the transversal size, we further assume that the degree of all vertices is 2, then the family constructed is unique upto bijection of vertex set.
\end{corollary}
\begin{proof}
It is easy to see that such a family is unique upto bijection of vertex set.
\end{proof}

\begin{corollary} Given any intersecting $k$-family $\cF$ with all vertices having degree 2 and any two $k$-sets intersect in exactly one vertex, then $\cF = \cM_k$. 
\end{corollary}
\begin{proof} For any intersecting family, if any two $k$-sets intersect in exactly one vertex then there is just one choice for a $k$-set at $l$'th stage (up to bijection of set of vertices). If we further force the condition that all vertices have degree 2, then such a $\cF$ must be $\cM_k$. 
\end{proof}

\begin{lemma} \label{lemma_dichotomy} Let $k > 1$. Suppose $\cF$ be an intersecting $k$-family of transversal size $k$ and minimal length. Then either $\cF$ has a vertex of degree 3 or $\cF = \cM_2$.
\end{lemma} 
\begin{proof}
Suppose there doesn't exist any vertex of degree 3. Then all vertex must have degree 2, for if there exists a vertex of degree 1, then the remaining $k-1$ vertices of the corresponding $k$-set cover $\cF$ which contradicts the assumption on trasnversal size.

Now suppose there exists two $k$-sets which intersect in more than one vertex, then the remaining $k-2$ vertices along with one of these 2 vertex is covering the whole family. So any two $k$-sets intersect in precisely one vertex. In particular, $\cF = \cM_k$. Transversal size considerations gives that $\cF = \cM_2$.
\end{proof}

\section{Some applications}

\begin{theorem} \label{main_result} q(4) = 9
\end{theorem}
\begin{proof}
Suppose $q(4) \leq 8$. Consider any intersecting 4-family $\cF$ of transversal size $4$ and minimal length. By the lemma \ref{lemma_dichotomy} it has at least one vertex of degree $ \geq 3$. Suppose $x$ has degree $> 3$. Define
$$\cF_x := \{B \in \cF| \,\, x \notin B\}$$ 

Then $\cF_x$ has length $\leq 4$ and transversal size = 3. This is not possible as it is easy to verify that any intersecting 4-family of length $\leq 4$ has transversal size atmost 2. Therefore degree of $x = 3$ and $\cF_x$ has length at least 5 and transversal size = 3. 

Suppose the length of $\cF_x$ is 5. We claim that then $\cF_x = \cM_4$ (upto bijection). To see this first note that if there exists any vertex of degree $\geq 3$ in $\cF_x$, then its transversal size will be $\leq 2$. On the other hand, if any vertex appears only once in the family, then in the corresponding $k$-set there must be a vertex with degree 3 or more. Thus every vertex of $\cF_x$ has degree 2. By corollary \ref{cor_uniqueness}, $\cF_x = \cM_4$ as claimed.

The above discussion implies that $\cF$ is given by adjoining three sets (all containing the common vertex $x$) to the family $\cF_x$. Since $\cF_x = \cM_4$ contains exactly 10 vertices each with degree 2, therefore total number of vertices in $\cF$ is 11. A simple double counting argument then gives that there are 10 vertices with degree 3 and only one with degree 2 (we use the fact that sum of degrees of all vertices is $8 \times 4 = 32$ and there are no vertices of degree 1 or $\geq 4$).

We will now show that there exists a 3-set which covers $\cF$. Among these 11 vertices, total pairs of vertices which is possible is $\binom{11}{2}$, out of which at most $\binom{4}{2} \cdot 8$ pairs occur in the family. So at least 7 pairs don't occur in the family. The unique vertex with degree 2 can contribute only $10 - (3+3) = 4$ pairs to these 7 pairs. So there exists at least 3 pairs of vertices, where both vertices have degree 3, which don't occur together in the family. Pick one of them (say) $\{a,b\}$. Then $\{a,b\}$ cover a length 6 subfamily of $\cF$. Remaining 2 members of $\cF$ must intersect in some vertex (say) $c$. Then $\{a,b,c\}$ cover $\cF$.

This contradicts the assumption that $\cF_x$ has length 5 and thus proves that $q(4) > 8$. To complete the proof, we have presented a family of length 9 in the final section.
\end{proof}

As another application, we give a different proof of the following result from \cite{Frankl},
\begin{corollary} $q(3) = 6$.
\end{corollary}
\begin{proof}
For if suppose there exists an intersecting 3-family $\cF$ of transversal size 3 of length 5, then by the lemma \ref{lemma_dichotomy} above it has at least one vertex $x$ of degree 3. We pick any vertex $y$ common to the 2 members of the family $\cF_x$ and then $\{x,y \}$ cover the family $\cF$. Since examples of intersecting 3-families with transversal size 3 are well known, this finishes the proof.
\end{proof}

We now prove a special property of the family $\cM_k$.
\begin{lemma} \label{lemma_disjoint} If $k$ is odd, then there exists $k$ disjoint transversals of the family $\cM_k$. 
\end{lemma}
\begin{proof} We enumerate the blocks of $\cM_k$ as $B_1, \ldots B_{k+1}$. Any transversal has length $(k+1)/2$ and since the degree of each vertex is 2, it corresponds to a partition of integers $1, \ldots, k+1$ into $(k+1)/2$ pairs. Two transversals are disjoint if and only if the corresponding partitions do not have any common pair. So to prove the statement, we need to find $k$ partitions of the set $\{1,2\ldots k+1\}$ into pairs, with no two partition having common pair.

Consider the set $\cS$ of all possible $\binom{k+1}{2} $ pairs. Let $(1,2), (3,4) \ldots (k, k+1)$ be the first partition. Remove these pairs from set $\cS$. Suppose we have picked $l$'th partition where $l < k$ (such that no 2 partitions intersect), we remove all the pairs selected so far from the set $\cS$. Now pick the $l+1$'th partition as follows: pick any pair containing $1$ (there must be some still in $\cS$ as $1$ can form $k$ pairs and only $l$ have been chosen so far). Iteratively pick the next pair by picking a pair containing the smallest integer which has so far not appeared in this partition.

Each step takes away $\frac{k+1}{2}$ pairs from $\cS$ so the algorithm will terminate after $k$'th step as we have exhausted all possible pairs.
\end{proof}

In particular the transversal size of the family of transversals of $\cM_k$ is $k$.

\begin{theorem} Let $k = 2^m - 1$ for any integer $m \geq 2$. Then there exists a uniform intersecing regular $k$-family such that degree of each vertex in the family is 3. Furthermore, the length of this family is $2k+1$. In particular the transversal size of the family is at least $(2k+1)/3$.
\end{theorem}
\begin{proof}
The proof is by induction on $m$. For $m=2$, projective plane of order 2 satisfies all the properties. Suppose the statement is true for $m - 1$. To construct such a family for $m$, we consider $\cM_k$ where $k = 2^{m} - 1$. By lemma \ref{lemma_disjoint} there exists $k$ disjoint transversals of $\cM_k$, which we call $T_1, T_2 \ldots T_k$. By induction, we can pick a family satisfying the statement of the theorem for $\frac{k-1}{2} = 2^{m - 1} - 1$ of length $k$ such that the vetex set of this family is different from that of $\cM_k$. 

Enumerating the blocks of this family as $B_1, \ldots B_k$, we consider the family of $k$-sets $\cG= \{T_i \sqcup B_i \,|\, i = 1, \ldots k\}$. Consider now the family $\cF = \cM_k \sqcup \cG$. It is easy to verify that $\cF$ satisfies all the properties mentioned in the theorem. 
\end{proof}

\subsection{An Example}

The following is an example of an intersecting 4-family with 9 blocks and having transversal size 4,

$$\begin{array}{ccc}
(1,2,3,4),& (1,5,6,7),& (2,5,8,9),\\
(3,6,8,10),&  (4,7,9,10),&  (1,8,9,11),\\
(2,6,7,11),&  (3,4,5,11),&  (1,2,5,10),
\end{array}$$

Somewhat surprisingly, $\cM_4$ is embedded in this family as first 5 blocks. This completes the proof of the theorem \ref{main_result}.

We thank Kaushik Majumder for suggesting the problem and various inputs.
%


\begin{thebibliography}{AAAAA}

\bibitem{Er-Lov} Erd\"{o}s P. and Lov\'{a}sz L., \textit{Problems and results on 3-chromatic hypergraphs and some related questions}, Colloquia Mathematica Societatis J\'{a}nos Bolyai, 10. Infinite and Finite sets, Keszthely (Hungary), 1973

\bibitem{Er} Erd\"{o}s P., \textit{On the combinatorial problems which I would most like to see solved}, Combinatorica, volume 1, issue 1, 25-42, 1981

\bibitem{Frankl} Frankl P., Ota K. and Tokushige N., \textit{Covers in uniform intersecting families and a counterexample to a conjecture of Lov\'{a}sz}, Journal of Combinatorial Theory, Series A, 74, 33-42 (1996)

\bibitem{Jeff} Kahn J., \textit{{On a problem of Erd\"{o}s and Lov\'{a}sz. II: n(r) = O(r)}}, Journal of the American Mathematical Society, volume 7, number 1, 125-143, 1994
\end {thebibliography}
\end{document}